\documentclass{article}

\usepackage{arxiv}

\usepackage[utf8]{inputenc} % allow utf-8 input
\usepackage[T1]{fontenc}    % use 8-bit T1 fonts
\usepackage{hyperref}       % hyperlinks
\usepackage{url}            % simple URL typesetting
\usepackage{booktabs}       % professional-quality tables
\usepackage{amsfonts}       % blackboard math symbols
\usepackage{nicefrac}       % compact symbols for 1/2, etc.
\usepackage{microtype}      % microtypography
\usepackage{lipsum}
\usepackage{amsmath}
\usepackage{color}
\usepackage{graphicx}
\usepackage{mathtools}
\usepackage{systeme}
\usepackage{amsthm}

\graphicspath{ {./images/} }
\newtheorem{lem}{Lemma}
\newtheorem{thm}{Theorem}
\newtheorem{app}{Approximation}
\newtheorem{rem}{Remark}
\newtheorem{example}{Example}

\title{Finite-Time stabilization of linear systems with unknown control direction via Extremum Seeking}

\author{
 Adriano Mele \\
  Consorzio CREATE\\
  via Claudio 21, 80125 Napoli \\
  \texttt{adriano.mele@unina.it} \\
   \And
 Gianmaria De Tommasi \\
  Dipartimento di Ingegneria Elettrica e Tecnologie dell'Informazione\\
  Università degli Studi di Napoli Federico II,\\
  Consorzio CREATE\\
  via Claudio 21, 80125 Napoli \\
  \texttt{detommas@unina.it} \\
  \And
 Alfredo Pironti \\
  Dipartimento di Ingegneria Elettrica e Tecnologie dell'Informazione\\
  Università degli Studi di Napoli Federico II,\\
  Consorzio CREATE\\
  via Claudio 21, 80125 Napoli \\
  \texttt{pironti@unina.it} \\
}

\begin{document}

\maketitle

\begin{abstract}
In this paper the finite-time stabilization problem is solved for a linear time-varying system with unknown control direction by exploiting a modified version of the classical extremum seeking algorithm. We propose to use a suitable oscillatory input to modify the system dynamics, at least in an \emph{average} sense, so as to satisfy a Differential Linear Matrix Inequality~(DLMI) condition which in turns guarantees that the system's state remains inside a prescribed time varying hyper-ellipsoid in the state space. The finite-time stability~(FTS) of the averaged dynamics implies the~FTS of the original system, as the distance between the original and the averaged dynamics can be made arbitrarily small by choosing a sufficiently high value of the dithering frequency used by the extremum seeking algorithm. An estimate of the necessary minimum dithering/mixing frequency is provided, and the effectiveness of the proposed finite-time stabilization approach is analysed by means of numerical examples.
\end{abstract}


%%%%%%%%%%%%%%%%%%%%%%%%%%%%%%%%%%%%%%%%%%%%%%%%%%%%%%%%%%%%%%
\section{Introduction}\label{sec:intro}
Extremum Seeking (ES) was originally introduced in~\cite{leblanc:1922} as a method to find (local) extrema of an unknown function, possibly the output of a dynamical system, which depends on one or more tunable parameters. The gist of this technique is to start from a rough estimate of the optimal parameters value and then exploit a sinusoidal perturbation to explore the unknown map around said estimate in order to move towards a local optimum.
% 
% During the 50-60's, \emph{extremum control} techniques which allowed to find the extremum of an unknown or uncertain map became popular, despite the lack of a rigorous mathematical treatment; indeed, 
A formal proof of the stability of ES applied to stable nonlinear systems with an unknown output map first appeared in the literature in 2000 (see~\cite{krstic:2000} and the references therein), which made use of a combination of averaging and singular perturbation theory.

A first attempt to extend this technique to simple linear marginally stable and unstable systems can be found in~\cite{zhang:2007}, where a method for tracking a target emitting a signal in the absence of any position measurement was proposed for autonomous vehicles. 
Although~\cite{zhang:2007} regarded the stability properties of the considered system as an obstacle for the optimization the output functional, the stabilization of the system can also be considered a goal by itself; in this view, modified ES algorithms that minimize Lyapunov-like functions have been proposed. 
In particular, a possible stabilizing~ES technique was originally introduced in~\cite{stankovich:2011}, where the authors analyzed the link between the trajectories of a system excited by a periodic, zero-average perturbation and the associated \emph{Lie bracket averaged} system~\cite{gurvits:1992}. In particular, it can be shown that the trajectories of the original system converge uniformly to those of the averaged system as the parameter~$\varepsilon$, linked to the frequency and the amplitude of the perturbation, tends to~0. Moreover, exploiting the notion of \emph{semi-global practical stability} introduced in~\cite{moreau:2000}, it can be shown that, if the Lie-bracket averaged system is globally uniformly asymptotically stable, then the original system is practically globally uniformly asymptotically stable for a sufficiently small value of~$\varepsilon$, i.e. its trajectories are confined in a~$\mathcal{O}(\varepsilon)$ neighborhood of the origin of the state-space. 
Based on that, in~\cite{scheinker:thesis, krstic:book} the authors analyse the stabilizing properties of the proposed ES scheme for a variety of systems (including~linear time varying and non-linear, non-affine in control systems) using different dithering signals. The proposed methodology is applied to the problem of tuning the quadrupole magnets and the bouncer cavities of a particle accelerator installed at the Los Alamos Neutron Science Centre.
A great advantage of this stabilization technique is that it is capable of dealing with systems whose control direction is unknown.

Inspired by these works, in the present article we try to extend these~ES stabilization results to a different kind of stability property, namely the Finite-Time Stability~(FTS) of linear dynamical systems~\cite{FTSbook}. 

Finite-Time stabilization is a concept linked to, but independent from, Lyapunov stabilization. In particular, a system is said to be~FTS with respect to a given time-horizon~$T$, an initial time instant~$t_0$, a positive-definite symmetric matrix~$R$ and a positive-definite symmetric matrix-valued function of time~$\Gamma(t)$ defined over the time interval~$\left[t_0\,, t_0+T\right]$, if the state trajectory starting from a point inside the hyper-ellipsoid defined by~$x_0^TRx_0 \le 1$ stays inside the time-varying hyper-ellipsoid defined by~$x^T(t)\Gamma(t)x(t)<1$.

The concept of FTS, originally introduced in the control literature in the 60's~\cite{dorato:1961, weiss:1967, Michel:1972}, has seen a renewed interest when efficient computational tools to solve algebraic~(LMI) and Differential~(DLMI) Linear Matrix Inequality problems became available, allowing to verify ``practical" FTS conditions~\cite{amato:1998, amato:2001} for linear time-varying~(LTV) systems. 
More recently, the FTS problem has been tackled for hybrid systems~\cite{zhao2008finite,amato2011finite,amato2013necessary} as well as in the stochastic framework~\cite{luan2010finite,yan2013finite,tartaglione2019annular}. Such increasing interest in~FTS and in the associated input-output notion (IO-FTS~\cite{Amato:TAC2012}) comes from the possibility of effectively adopt FTS concepts to enforce specific quantitative requirements on the transient of the closed-loop response of a control system~\cite{amato2014input,ariola2018hybrid}. 

The main idea of the present work is to apply the FTS stabilization techniques available in the literature to the Lie bracket averaged model obtained applying the ES controller to a~LTV plant, possibly with unknown control direction, and then exploit the uniform convergence of the trajectories of the original system to those of the averaged one to draw conclusions on its FTS properties.

The rest of the paper is organized as follows:
 Section~\ref{sec:preliminary} gives an overview of the mathematical background, including the main concepts of Lie bracket averaging and Finite-Time Stabilization; Section~\ref{sec:FTS} presents the application of the Finite-Time control techniques presented in~\cite{FTSbook} to the Lie-Bracket averaged system. In section~\ref{sec:bounds} some practical indications are given for the choice of the dithering frequency of the ES scheme, and finally Section~\ref{sec:examples} shows some examples of the application of the proposed technique. Section~\ref{sec:conclusion} concludes the article.

%%%%%%%%%%%%%%%%%%%%%%%%%%%%%%%%%%%%%%%%%%%%%%%%%%%%%%%%%%%%%%
\section{Background overview}\label{sec:preliminary}

In this section, some preliminary concepts are introduced. In particular, in Section~\ref{sec:liebracket} the notion of \emph{Lie bracket averaged} system associated to a dynamical system subject to periodic inputs is presented; in Section~\ref{sec:convtraj} the notion of systems with \emph{converging trajectories}, i.e. state trajectories whose distance can be made arbitrarily small by acting on a parameter, is discussed. Finally, in Section~\ref{sec:FTSintro} the concept of FTS of a linear system is recalled, together with some necessary and sufficient conditions.

\emph{Notation:}
in the following, $|| \cdot ||$ denotes the norm of a matrix, while $|\cdot|$ denotes the norm of a vector. Moreover, given two symmetric matrices $M$ and $N$, $M \prec 0$ indicates that $M$ is negative-definite, $M \succ 0$ that it is positive-definite,~$M \preceq 0$ that it is negative-semidefinite, $M \succeq 0$ that it is positive-semidefinite; $M\prec N$ is equivalent to~$M-N\prec 0$ and similarly for $M\preceq N$, $M\succ N$, $M\succeq N$. $Dom (\cdot)$ indicates the domain of a function. Finally, $[ f(\cdot) ]_{t_0}^t := f(t)-f(t_0)$.

\subsection{Lie bracket averaging}\label{sec:liebracket}

Consider a system in the general form

\begin{equation}\label{eq:NLsystem}
\begin{aligned}
    \dot{x}(t) &= \sum_{i=1}^{m_1}{b_i(x)u_i(t)} 
     + \frac{1}{\sqrt{\varepsilon}}\sum_{i=1}^{m_2}{\hat{b}_i(x)\hat{u}_i(t,\theta)} \\
     &x(t_0) = x_0\,,
\end{aligned}
\end{equation}

where the functions $\hat{u}_i(t\,,\theta)$ are $T_u$-periodic in $\theta = t/\varepsilon$ with zero-average on the period $T_u$. 
The \emph{Lie-bracket averaged} system~\cite{krstic:book} associated to~\eqref{eq:NLsystem} is

\begin{equation}\label{eq:liebracketavg}
\begin{aligned}
    \dot{\bar{x}}(t) &= \sum_{i=1}^{m_1}{b_i(\bar{x})u_i(t)} 
     + \frac{1}{T_u}\sum_{i=1\,, i<j}^{m_2}{ [\hat{b}_i\,, \hat{b}_j](\bar{x}) \nu_{ij}(t)} \\ 
     &\bar{x}(t_0) = x(t_0)\,,
\end{aligned}
\end{equation}

where $\nu_{ij}(t)$ is defined as
\begin{equation*}\label{eq:nu}
     \nu_{ij}(t) = \int_0^{T_u}{\int_0^{\theta}{\hat{u}_i(t,\sigma)\hat{u}_j(t,\theta)d\sigma d\theta}}\,,
\end{equation*}

and $[\hat{b}_i\,, \hat{b}_j](x)$ is the standard \emph{Lie bracket} of $\hat{b}_i(x)\,, \hat{b}_j(x)$
\begin{equation*}
     [\hat{b}_i, \hat{b}_j](x) = \frac{\partial \hat{b}_j(x)}{\partial x} \hat{b}_i(x) - \frac{\partial \hat{b}_i(x)}{\partial x} \hat{b}_j(x)
\end{equation*}

% As it will be discussed in Section~\ref{sec:convtraj}, the interest in this approach resides in the fact that it can be shown that the state trajectories of the ``original" and averaged systems are close one to each other of a quantity that shrinks with $\varepsilon \rightarrow 0$.

Note that this definition holds for all the integer multiples $nT_u$ of the period $T_u$.

\subsection{Converging trajectories property} \label{sec:convtraj}

The basic hypothesis that underlies the method proposed in this work is the \emph{convergence of trajectories}~\cite{moreau:2000} for the original and the average systems. Consider a generic systems 

\begin{equation}\label{eq:sys}
\dot{x}(t) = f(t\,,x)\,,
\end{equation}

and its perturbed counterpart

\begin{equation}\label{eq:syseps}
\dot{x}^{(\varepsilon)}(t) = f^{(\varepsilon)}(t\,,x^{(\varepsilon)})\,,
\end{equation}

where the superscript $(\varepsilon)$ indicates the dependence of the system dynamics on the small parameter $\varepsilon$ (e.g. as in~\eqref{eq:NLsystem}). Denote by $\Phi(t,t_0,x_0)$ and $\Phi^\varepsilon(t,t_0,x_0)$ the solutions of~\ref{eq:sys} and~\ref{eq:syseps} respectively passing through the point $x_0$ at $t=t_0$.
Systems~\ref{eq:sys}-\ref{eq:syseps} are said to have \emph{converging trajectories} if, $\forall~\hat{T}>0$, $\forall~K$ compact subset of $\mathbb{R}^n$ such\linebreak 
that $\left\{ t\in[t_0\,,t_0+\hat{T}], x\in K \right\} \in Dom(\Phi)$ and $\forall~\Delta>0$, $\exists~\varepsilon^* >0 : \forall~t_0 \in \mathbb{R}, \forall~x_0 \in K, \forall \varepsilon \in (0,\varepsilon^*)$
\[
    \left|\Phi^{(\varepsilon)}(t,t_0,x_0)- \Phi(t,t_0,x_0)\right| < \Delta\,, \qquad \forall t \in [t_0, t_0+\hat{T}]\,.
\]

In particular, this property holds for a given system in the form~\eqref{eq:NLsystem} and the corresponding Lie bracket averaged system~\eqref{eq:liebracketavg}~\cite{krstic:book,moreau:2000,durr:2013}, in the sense that, given the period $T_u<0$, $n\in\mathbb{N}$, then $\exists~\varepsilon^* >0 : \forall~\varepsilon \in (0,\varepsilon^*)$
\[
    \max_{t\in[0,nT_u]}{|x(t)-\bar{x}(t)|} \le \Delta(nT_u\,,\varepsilon)\,,
\]
where $\Delta \rightarrow 0$ as $\varepsilon \rightarrow 0$.

\subsection{Finite-Time Stabilization (with ellipsoidal domains)}\label{sec:FTSintro}

We now recall the definition of FTS~\cite{FTSbook} of a LTV system. 
Generally speaking, given a positive-definite, symmetric matrix $R$ and a positive definite, symmetric matrix-valued function $\Gamma(t)$ defined over a time interval $[t_0, t_0+T]$, an autonomous LTV system in the form
\begin{equation}\label{eq:LTV}
\dot{x}(t) = A(t)x(t)\,, \qquad x(t_0) = x_0\,,
\end{equation}
is said to be FTS with respect to $\left(t_0\,,T\,,\Gamma(\cdot)\,,R\right)$ iff, by definition, 
\begin{equation*}\label{eq:FTS}
x_0^T R x_0 \le 1 \implies x^T(t)\Gamma(t)x(t) < 1\,,  t \in \left[ t_0\,, t_0 + T \right]\,.
\end{equation*}
Note that, for this definition to be well-posed, it must hold true that $\Gamma(t_0) \prec R$. 

In~\cite[Thm.~2.1]{FTSbook} several equivalent conditions are given in order to assess if a system in the form~\eqref{eq:LTV} is FTS. 
These conditions can be extended to the state feedback closed-loop system 
\begin{subequations}\label{eq:systemABCD2}
\begin{equation}
\dot{x}(t) = A(t)x(t) + B(t)u(t) 
\end{equation}
\begin{equation} \label{eq:statefeedback}
u = K(t) x(t)
\end{equation}
\end{subequations}

In particular, the system~\eqref{eq:systemABCD2} is said to be Finite Time Stabilizable by a linear static state feedback control law wrt~$\left(t_0\,,T\,,\Gamma(\cdot)\,,R\right)$ iff~\cite[Thm.~3.1]{FTSbook}
\begin{equation}\label{eq:FTScond}
\left\{
\begin{aligned}
  -\dot{Q}(t) &+ Q(t)A^T(t) + A(t)Q(t) + L^T(t)B^T(t) \\
  &\begin{aligned}
    + B(t)L(t) & \prec  0 \\
  Q(t) & \prec  \Gamma^{-1}(t)\,, & t \in [t_0\,, t_0+T] \\
  Q(t_0) & \succ R^{-1}\,, & t \in [t_0\,, t_0+T]
  \end{aligned}
\end{aligned}
\right.
\end{equation}

If condition~\eqref{eq:FTScond} is satisfied for some $Q(t)$ and $L(t)$, then the controller gain that FT-stabilizes the system is given by
\begin{equation}\label{eq:controlgain}
    K(t) = L(t)Q^{-1}(t)\,.
\end{equation}

%%%%%%%%%%%%%%%%%%%%%%%%%%%%%%%%%%%%%%%%%%%%%%%%%%%%%%%%%%%%%%
\section{Finite-Time stabilization via Extremum-Seeking and Lie bracket averaging}\label{sec:FTS}
We are now ready to introduce the main contribution of this paper, namely the finite-time stabilization of single-input LTV systems via extremum seeking.

Let us consider a single-input LTV system in the form
\begin{subequations}\label{eq:LTVcase}
\begin{equation}\label{eq:systemABCD}
\dot{x}(t) = A(t)x + B(t)u\,, \qquad x(t_0) = x_0\,,
\end{equation}
where $B(t)\in \mathbb{R}^{n\times 1}$, and the following control law
\begin{equation} \label{eq:ESinput}
u(t) = \alpha  \sqrt{\omega} \cos(\omega t) - k \sqrt{\omega} \sin(\omega t) x^T \Pi(t) x\,,
\end{equation}
\end{subequations}
where the $\Pi(t)$ matrix is assumed to be symmetric and positive-definite.
According to~\eqref{eq:liebracketavg}, the Lie-bracket averaged system corresponding to~\eqref{eq:LTVcase} is
\begin{equation}\label{eq:avgLTV}
    \dot{\bar{x}}(t) = A(t)\bar{x} - k\alpha B(t)B^T(t)\Pi(t)\bar{x}\,.
\end{equation}
where $\frac{1}{\omega}$ plays the role of the small parameter $\varepsilon$.

If the averaged system can be finite-time stabilized, then the converging trajectories property stated in Section~\ref{sec:convtraj} ensures that the state of the closed-loop system will be drawn towards the desired region of the phase space as $\omega \rightarrow \infty$.

Clearly, the FTS of the averaged system does not automatically imply the FTS of the closed loop system, since the oscillations of $x(t)$ around the averaged trajectory $\bar{x}(t)$ could still violate the requirement $x^T(t)\Gamma(t)x(t)<1$. However, thanks to the convergence of trajectories, for any given value of $\Delta$ (i.e. the maximum allowed distance between $x$ and $\bar{x}$) it is always possible~\cite{durr:2013, gurvits:1992} to find a minimum frequency $\omega^*$ such that $|x-\bar{x}|<\Delta \quad \forall \omega>\omega^*$. 
Thus, proper FT stabilization of system~\eqref{eq:LTVcase} can be achieved by FT stabilizing the averaged system~\eqref{eq:avgLTV} with respect to an opportune smaller region in state-space, and then by choosing a frequency $\omega$ such that the distance between the boundaries of these two regions is not exceeded by the distance of the state trajectory from its average (see Fig.~\ref{fig:lemma1}).
\begin{figure}
    \centering
    \includegraphics[width=0.55\linewidth]{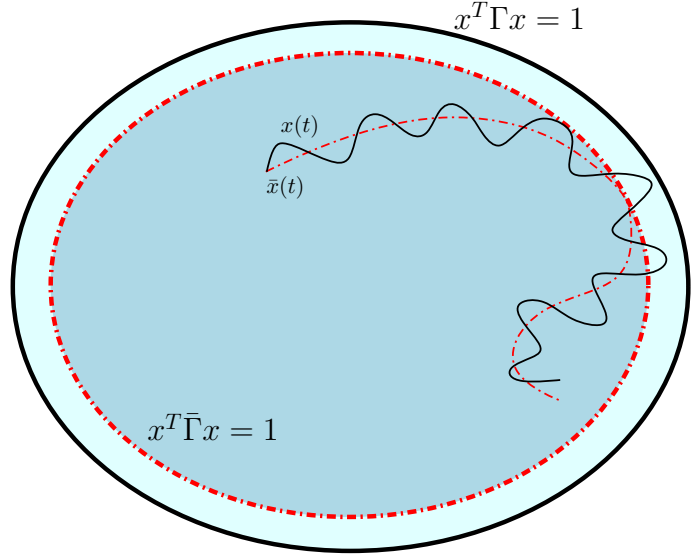}
    \caption{The aim of the proposed technique is to finite-time stabilize the averaged dynamics $\bar{x}(t)$ (dash-dotted red trajectory) with respect to $\bar{\Gamma}(t)$ (the ellipse defined by $x^T \bar{\Gamma} x = 1$ is shown by the dash-dotted red line). If the distance between $x(t)$ and $\bar{x}(t)$ is bounded, the matrix $\bar{\Gamma}(t)$ can be chosen so as to guarantee that the state trajectory $x(t)$ is FTS with respect to $\Gamma(t)$ by choosing $\bar{\Gamma}(t)$ appropriately.}
    \label{fig:lemma1}
\end{figure}

This observation leads us to establish the following lemma.

%%%%% Lemma 1

\begin{lem}\label{lemma1}
Consider the two hyper-ellipses defined by
\begin{equation}\label{eq:ellipse}
    x^T \Gamma x = 1\,, \quad y^T \Gamma y = r^2\,,
\end{equation}
where $\Gamma$ is a $n \times n$ positive-definite matrix and $x,y \in \mathbb{R}^n$. Then, the minimum distance between the two hyper-ellipses is given by
\[
    (1-r) \min_i{\gamma_i}\,,
\]
where $\left\{ \frac{1}{\gamma_i^2} \right\}$ is the set containing all the eigenvalues of~$\Gamma$.

\end{lem}

\begin{proof}
First of all, assume, without loss of generality, that~$\Gamma = diag\left( \frac{1}{\gamma_1^2}, \frac{1}{\gamma_2^2}, ..., \frac{1}{\gamma_n^2} \right)$
Indeed, since $\Gamma$ is positive definite, there always exists an orthonormal matrix~$V$ such that $\Gamma V = VD$,
with $D$ diagonal. Now, consider the quantity
\[
    1-r^2 =
    \sum_{i=1}^n \frac{x_i^2 - y_i^2}{\gamma_i^2} = 
    \sum_{i=1}^n (x_i-y_i)\frac{x_i + y_i}{\gamma_i^2} \,.
\]
By applying the Cauchy-Schwartz inequality, we find
\begin{equation*}
\begin{aligned}
    1-r^2  
    & \le
    \left[ \sum_{i=1}^n (x_i-y_i)^2 \right]^{1/2} 
    \left[ \sum_{i=1}^n \frac{(x_i + y_i)^2}{\gamma_i^4} \right]^{1/2} \\
    & \le 
    \left[ \sum_{i=1}^n (x_i-y_i)^2 \right]^{1/2} 
    \left[ \sum_{i=1}^n \frac{(x_i + y_i)^2}{\gamma_i^2} \right]^{1/2} \frac{1}{\min_i (\gamma_i)} \\
    & =  
    \frac{|x-y|}{\min_i (\gamma_i)}
    \left[ \sum_{i=1}^n \frac{(x_i + y_i)^2}{\gamma_i^2} \right]^{1/2}\,.
\end{aligned}
\end{equation*}
Then, by applying the triangular inequality to the term in square brackets
\begin{equation*}
\begin{aligned}
    1-r^2 &\le
    \frac{|x-y|}{\min_i (\gamma_i)}
    \left\{
    \left[ \sum_{i=1}^n \frac{x_i^2}{\gamma_i^2} \right]^{1/2}
    +
    \left[ \sum_{i=1}^n \frac{y_i^2}{\gamma_i^2} \right]^{1/2}
    \right\} \\
    &= \frac{|x-y|}{\min_i (\gamma_i)} (1+r)\,,
\end{aligned}
\end{equation*}
hence
\[
   |x-y| \ge (1-r) \min_i{\gamma_i}\,.
\]

Observe that the equality is attained when $x$ and $y$ are aligned with the minor semi-axes of the hyper-ellipsoids defined in~\eqref{eq:ellipse}.
%%% Serve per assicurare che il bound sia un minimo
% 
To conclude the proof, we observe that, for a generic $A$ matrix, we can consider the distance
\[
    |V(x-y)| \le ||V|| \cdot |x-y| = |x-y|\,,
\]
With this choice, $(Vx)^T \Gamma (Vx) = x^T D x$
and
\[
    |x-y| \ge |V(x-y)| \ge (1-r) \min_i{\gamma_i}\,.
\]

\end{proof}

%%%% Lemma 2

\begin{lem}\label{lemma2}
For $x(t), \bar{x}(t) \in \mathbb{R}^n$ for a given $\Delta>0$, assume
\[
    |x(t)-\bar{x}(t)| < \Delta\,.
\]

Consider two symmetric, positive definite matrix valued functions of time $\Gamma(t)$ and $\bar{\Gamma}(t)$ such that, $\forall~t \in [t_0\,, t_0+T]$
\[
    \bar{\Gamma}(t) = \frac{1}{r(t)^2}\Gamma(t)\,.
\]

If the inequality $\bar{x}^T(t) \bar{\Gamma}(t) \bar{x}(t) < 1$
holds on the time interval $t \in [t_0, t_0+T]$, then, on the same interval
\[
    x^T(t) \Gamma(t) x(t) < 1 \quad \forall r(t) \le 1-\frac{\Delta}{\min_i(\gamma_i(t))}\,,
\]

where $\left\{\frac{1}{\gamma_i(t)^2}\right\}$ are the eigenvalues of $\Gamma(t)$ at time $t$.

\end{lem}

\begin{proof}

The result is obtained immediately by applying Lemma~\ref{lemma1} at each time instant $t$.

\end{proof}

\begin{rem} For well-posedness, it must hold that $\Delta < \min_{i\,,t} \left\{\gamma_i(t) \right\} \forall t\in[t_0, t_0+T]$.
\end{rem}

% \begin{rem} The equality is attained when $x$ is aligned with the minor semi-axis of the hyper-ellipsoid defined by $\Gamma(t)$.
% \end{rem}

Using Lemma~\ref{lemma2}, we can now state the following result.

\begin{thm}\label{thm1}
Consider the LTV system~\eqref{eq:LTVcase} and its averaged version~\eqref{eq:avgLTV}. Suppose that the dithering/mixing frequency $\omega$ is chosen so that $|x(t)-\bar{x}(t)|<\Delta$ for $t \in [t_0\,, t_0+T]$.

If the following DLMI condition is satisfied for some~$Q(t)\,,k\,,\alpha$ and~$\Pi(t)$.

\begin{equation*}
\left\{
\begin{aligned}
  -&\dot{Q}(t) + Q(t)A(t) + A^T(t) Q(t) - k\alpha Q(t) \Pi(t) B(t) B^T(t) \\
&\begin{aligned}
  - k\alpha B(t) B^T(t) \Pi(t) Q(t) &\prec 0  &\forall~t \in [t_0\,, t_0+nT] \\
  Q(t) &\prec  \bar{\Gamma}^{-1}(t) &\forall~t \in [t_0\,, t_0+nT] \\
  Q(t_0) &\succ R^{-1} 
\end{aligned}  
\end{aligned}
\right.
\end{equation*}

where $\bar{\Gamma}(t)$ is the matrix-valued function of time defined in Lemma~\ref{lemma2}
% , i.e.
%
% \[
%     \bar{\Gamma}(t) = \frac{1}{r(t)^2}\Gamma(t)\,,
% \]
%
with
\[
   r(t) \le 1-\frac{\Delta}{\min_i(\gamma_i(t))}\,, 
\]
then the closed-loop system~\eqref{eq:LTVcase} is FTS with respect to~$\left(t_0\,,T\,,R\,,\Gamma(t)\right)$ for the same values of~$k\,,\alpha\,,$ and~$\Pi(t)$.

% ~\eqref{eq:FTScond}-\eqref{eq:controlgain} for
\end{thm}

\begin{proof}
Comparing~\eqref{eq:systemABCD2} to~\eqref{eq:avgLTV}, we choose
\begin{equation}\label{eq:controlgain2}
    K(t) = -k\alpha B^T(t)\Pi(t)\,.
\end{equation}

For Lemma~\ref{lemma1}, the FTS of the averaged system~\eqref{eq:avgLTV} with respect to~$\left(t_0\,,T\,,R\,,\bar{\Gamma}(t)\right)$ implies the FTS of the closed loop system~\eqref{eq:LTVcase} with respect to $\left(t_0\,,T\,,R\,,\Gamma(t)\right)$.

Condition~\eqref{eq:ES-FTS} is immediately obtained combining~\eqref{eq:FTScond}, \eqref{eq:controlgain}, \eqref{eq:controlgain2} and substituting $\Gamma(t)$ with $\bar{\Gamma}(t)$ in the FTS problem formulation for the averaged system.
\end{proof}

\begin{rem} Note that~$\Delta$ must be small enough so that the well-posedness condition $\bar{\Gamma}(t_0)<R$ is still satisfied. 
\end{rem}

We have obtained the following DLMI problem for the averaged system
\begin{equation}\label{eq:ES-FTS}
\left\{
\begin{aligned}
  -&\dot{Q}(t) + Q(t)A(t) + A^T(t) Q(t) -
  k\alpha Q(t) \Pi(t) B(t) B^T(t) \\
  &
  \begin{aligned}
  - k\alpha B(t) B^T(t) \Pi(t) Q(t)] 
  &\prec 0 &\forall~t \in [t_0\,, t_0+nT]\\
  Q(t) &\prec  \bar{\Gamma}^{-1}(t) &\forall~t \in [t_0\,, t_0+nT] 
    \\
  Q(t_0) &\succ R^{-1} 
    \end{aligned} 
\end{aligned}
\right.
\end{equation}

However, this problem is still nonlinear, as it contains the product of the design parameters $k\,,\alpha\,,\Pi(t)$ and~$Q(t)$. 
%
%In order to obtain a solvable problem, 
To solve it, observe~\cite{scheinker:2013} 
% we recall the observation (found in the original formulation of the stabilizing ES algorithm~\cite{scheinker:2013})
that the term $k\alpha BB^T\Pi \bar{x}$ 
% that appears
in~\eqref{eq:avgLTV} is proportional to the gradient of the Lyapunov-like function $V(t\,,x) = x^T\Pi(t) x$, evaluated for $x = \bar{x}$. This term is weighted by the positive semi-definite matrix~$B(t)B(t)^T$. If the product $k\alpha$ is large enough, under a condition of persistency of excitation of $B(t)$, this gradient term dominates the $A(t)\bar{x}$ term, and the trajectory of the averaged system evolves according to a gradient descent of $V(t\,,\bar{x})$.
% the quadratic function $V(t\,,\bar{x})$. 
% 
%
According to the definition of FTS, we want to keep the quantity $x(t)^T\Gamma(t) x(t)$ below $1$, therefore it makes sense to choose $\Pi(t) = \Gamma(t)$. This choice will also turn useful in the calculations of Section~\ref{sec:bounds}. We can then perform a scan in the product $k\alpha$ in order to find a solution in terms of $Q(t)$. 
% It is worth to remark that this algorithm does not provide the parameters $k$ and $\alpha$ individually; here, we choose $k=\alpha=\sqrt{k\alpha}$ (for more details, see ex.~\ref{ex:Ex1}).
%%% From ex. 1
It is worth remarking that the proposed technique gives no particular prescription on how to tune these parameters, as the averaged system dynamics only depends on their product. However, their choice can influence the stability properties of the original system, the amplitude of the oscillations and the capability of the algorithm of escaping local minima of $V(x)$. 
For a discussion on the choice of~$k$ and $\alpha$, see~\cite[Sec.~1.3]{krstic:book}.

%%%%%%%%%%%%%%%%%%%%%%%%%%%%%%%%%%%%%%%%%%%%%%%%%%%%%%%%%%%%%%
\section{Practical choice of the dithering frequency}\label{sec:bounds}

As mentioned in the previous sections, the original and averaged system exhibit so-called \emph{converging trajectories}. In particular, it can be shown that given a distance $\Delta$, it is always possible to find a minimum frequency $\omega^*$ such that the distance $|x(t) - \bar{x}(t)|$ is smaller than $\Delta$ for all $\omega > \omega^*$. This means that, once the dithering frequency has been chosen such that the condition $\omega > \omega^*$ is satisfied and the control matrix-valued function $\Pi(t)$ has been fixed, Theorem~\ref{thm1} can be applied to find the values of the design parameters $k$ and $\alpha$ that guarantee the FTS of a system in the form~\eqref{eq:LTVcase} by means of the equivalent FTS problem formulated in terms of its autonomous Lie-bracket averaged counterpart~\eqref{eq:avgLTV}. 
% Questa forse è una ripetizione, ma male non fa
% 
We now turn our attention to the problem of finding an estimate of the minimum  dithering frequency needed for this modified extremum-seeking algorithm.

% One possible way to draw an explicit expression for $\omega^*$ is partially given in~\cite[thm. 1 - App. B]{durr:2013}. 

For simplicity we will consider the case where the $B(t)$ matrix is a constant of unknown sign, say $B(t)=B$. Moreover, let us fix $\Pi(t) = \Gamma(t)$ and assume $x(t_0) = \bar{x}(t_0) = x_0$ (note that the dithering signal can always be chosen so as to be $0$ at $t=t_0$). 

Direct integration of~\eqref{eq:LTVcase} gives
\begin{equation*}
\begin{aligned}
    x(t) & =  x_0 + \int_{t_0}^t A(\tau)x(\tau)d\tau
            + \frac{\alpha}{\sqrt{\omega}}B 
            \left[ 
                \sin(\omega \tau)
            \right]_{t_0}^t \\
        &- \int_{t_0}^t Bk\sqrt{\omega} \sin(\omega \tau) 
            \left( 
                x^T(\tau) \Gamma(\tau) x(\tau) 
            \right) d\tau\,,
\end{aligned}
\end{equation*}

Integrating by parts the last term, using again~\eqref{eq:LTVcase}, the fact that $x^T(t)\Gamma(t)B = B^T \Gamma(t) x(t)$ (it is scalar) and applying standard trigonometric identities we have (time dependencies are dropped for clarity)
\begin{equation}\label{eq:boundexact}
\begin{aligned}
    x(t) &=  x_0 + \int_{t_0}^t 
            \left[
                A - k \alpha B B^T \Gamma 
            \right]x d\tau \\
            &+ \frac{\alpha}{\sqrt{\omega}}B 
            \left[ 
                \sin(\omega \tau)
            \right]_{t_0}^t
            + \frac{k}{\sqrt{\omega}}B 
            \left[ 
                \cos(\omega \tau) x^T \Gamma x
            \right]_{t_0}^t  \\
            &- \int_{t_0}^t \frac{k}{\sqrt{\omega}}B \cos(\omega \tau) 
            \left( 
                x^T \dot{\Gamma} x
            \right)d\tau    \\ 
            &- \int_{t_0}^t \frac{2k}{\sqrt{\omega}}B \cos(\omega \tau) 
            \left( 
                x^T \Gamma A x 
            \right)d\tau  \\
            &- \int_{t_0}^t k\alpha \cos(2\omega \tau) 
            B B^T \Gamma x
            d\tau    \\
            &- \int_{t_0}^t k^2 \sin(2\omega \tau) 
            B B^T \Gamma x
            \left(
                x^T \Gamma x
            \right) d\tau\,.
\end{aligned}
\end{equation}

This expression is \emph{exact}. In particular, one possibility to find a lower bound on the dithering frequency $\omega$ would be to integrate by parts the terms depending on $2\omega\tau$ that appear on the last row of~\eqref{eq:boundexact}
\begin{equation*}\label{eq:partintegrals}
\begin{aligned}
    \int_{t_0}^t &k\alpha \cos(2\omega \tau) 
            B B^T \Gamma x
            d\tau \\
    & =  \left[ 
        \frac{k \alpha}{2\omega} B B^T \Gamma x \sin(2\omega \tau)
    \right]_{t_0}^t \\
    &- \int_{t_0}^t \frac{k \alpha}{2\omega} B B^T \sin(2\omega \tau)
        \left[
            \dot{\Gamma}x + \Gamma \dot{x}
        \right] d\tau  \\
    \int_{t_0}^t &k^2 \sin(2\omega \tau) 
            B B^T \Gamma x
            \left(
                x^T \Gamma x
            \right) d\tau \\
    & = \left[ 
        - \frac{k^2}{2\omega} B B^T \Gamma x \left(x^T \Gamma x\right) \cos(2\omega \tau)
    \right]_{t_0}^t     \\
    & + \int_{t_0}^t \frac{k^2}{2\omega} B B^T \cos(2\omega \tau)
        % \left[
            \left( \dot{\Gamma}x + \Gamma \dot{x} \right)
            \left( x^T\Gamma x \right) d\tau \\
     & + \int_{t_0}^t \frac{k^2}{2\omega} B B^T \cos(2\omega \tau)
            \Gamma x 
            \left( x^T\dot{\Gamma} x + 2x^T \Gamma \dot{x} \right)
         d\tau\,,
\end{aligned}
\end{equation*}
to obtain, along the lines of ~\cite[Thm.~1]{durr:2013}, an expression in the form
\begin{equation*}
    x-\bar{x} =\int_{t_0}^t 
            \left[
                A - k \alpha B B^T \Gamma 
            \right] (x - \bar{x}) d\tau
        + \sum_i R_i\,,
\end{equation*}
where each remainder term $R_i$ satisfies
\begin{equation*}
    |R_i| \le \frac{c_i}{\sqrt{\omega}}\,,
\end{equation*}
for some constant $c_i$ independent of $t_0$ and $x_0$ and for $\omega$ large enough, under some (reasonable) assumptions. % on the terms which involve derivatives. 
Then, the Gronwall-Bellman lemma can be applied to obtain an upper bound on the distance between the actual and averaged trajectories, which can be made arbitrarily small by increasing $\omega$.
% 
% \textcolor{red}{
However, the need for several partial integrations leads to cumbersome calculations, and to a result which is not readily interpretable. Moreover, the exploitation of the Gronwall-Bellman lemma easily leads to very conservative estimates.
% while being exact, this approach suffers from two main disadvantages:
% \begin{itemize}
%     \item[i.] the need for several partial integrations leads to cumbersome calculations
%     \item[ii.] the exploitation of the Gronwall-Bellman lemma easily leads to very conservative estimates
% \end{itemize}
% }
Hence, we propose to exploit the intrinsic \emph{time-scale separation} property of the algorithm in order to draw an approximate expression for $x(t)-\bar{x}(t)$. This leads us to invoke the following approximation.

% \textcolor{red}{
\begin{app}\label{app1}
    Exploit the \emph{time-scale separation} property of ES, and assume that the oscillations of the dithering and mixing terms vary on a much faster scale than the other terms appearing in the integrals of~\eqref{eq:boundexact}.
\end{app}
% }

The whole ES method is based on the implicit assumption that all the terms in the right-hand side of~\eqref{eq:boundexact} but the ones related to the average dyamics, i.e.
$
    x_0 + \int_{t_0}^t [A-k\alpha B B^T \Gamma] x d\tau\,,
$
vanish for $\omega \rightarrow \infty$. Hence, a "safe" approximation is to assume everywhere that, for a generic function of time $f(t)$
\begin{equation*}
\begin{aligned}
    \int_{t_0}^t &f(\tau) \sin(\omega \tau) d\tau 
    = \\
    &\left[\frac{f(\tau)}{\omega}\cos(\omega \tau)\right]_{t_0}^t
    +
    \int_{t_0}^t \frac{\dot{f}(\tau)}{\omega} \cos(\omega \tau) d\tau \\
    &\cong
    \left[\frac{f(\tau)}{\omega}\cos(\omega \tau)\right]_{t_0}^t\,,
\end{aligned}
\end{equation*}
i.e. $\dot{f}(t) << \omega$. This leads to approximate~\eqref{eq:boundexact} as:
\begin{equation}\label{eq:bound3}
    \begin{aligned}
    x(t) &\cong  x_0 + \int_{t_0}^t 
            \left[
                A - k \alpha B B^T \Gamma 
            \right]x d\tau \\
            & + \frac{\alpha}{\sqrt{\omega}}B 
            \left[ 
                \sin(\omega \tau)
            \right]_{t_0}^t 
            + \frac{k}{\sqrt{\omega}}B 
            \left[ 
                \cos(\omega \tau) x^T \Gamma x
            \right]_{t_0}^t  \\
            &- \frac{k}{\omega \sqrt{\omega}}B  
            \left[ 
                x^T \dot{\Gamma} x \sin(\omega \tau)
            \right]_{t_0}^t  \\
            &+ \frac{2k}{\omega\sqrt{\omega}}B 
            \left[
                x^T \Gamma A x \sin(\omega \tau)
            \right]_{t_0}^t \\
             &- \frac{k\alpha}{2\omega} B B^T 
            \left[
                \Gamma x \sin(2\omega \tau) 
            \right]_{t_0}^t \\
            &+ \frac{k^2}{2\omega} B B^T 
            \left[
            \Gamma x
            \left(
                x^T \Gamma x
            \right) \cos(2\omega \tau)
            \right]_{t_0}^t\,.
    \end{aligned}
\end{equation}

\begin{app}\label{app2}
    Since we are looking for a relatively large dithering frequency, we neglect the highest order terms in~$1/\sqrt{\omega}$ (i.e., those with $\omega\sqrt{\omega}$ at the denominator).
\end{app}

This leads to
\begin{equation}\label{eq:bound4}
    \begin{aligned}
    x(t) &\cong  x_0 + \int_{t_0}^t 
            \left[
                A - k \alpha B B^T \Gamma 
            \right]x d\tau \\
            & + \frac{\alpha}{\sqrt{\omega}}B 
            \left[ 
                \sin(\omega \tau)
            \right]_{t_0}^t 
            + \frac{k}{\sqrt{\omega}}B 
            \left[ 
                \cos(\omega \tau) x^T \Gamma x
            \right]_{t_0}^t  \\
             &- \frac{k\alpha}{2\omega} B B^T 
            \left[
                \Gamma x \sin(2\omega \tau) 
            \right]_{t_0}^t \\
            &+ \frac{k^2}{2\omega} B B^T 
            \left[ \Gamma x
            \left(
                x^T \Gamma x
            \right) \cos(2\omega \tau)
            \right]_{t_0}^t\,.
    \end{aligned}
\end{equation}

\begin{app}\label{app3}
    % For $\omega \rightarrow \infty$, $|x(t)-\bar{x}(t)|<\Delta$ and $x^T \Gamma(t) x < 1$. Assume that this conditions is satisfied.
    Assume $x^T \Gamma(t) x < 1$.
\end{app}

If $\omega$ is large enough, $|x(t)-\bar{x}(t)|<\Delta$ and the assumptions of Thm.~\ref{thm1} are satisfied. In turn, this implies that the FTS condition is satisfied for the controlled system, and thus approx.~\ref{app3} holds. 
% \textcolor{red}{
(see also the similar argument used in~\cite[p.1550]{scheinker:2013}).
% }

\begin{rem}
Intuitively, if $|B(t)|,k,\alpha$ are of order $\approx 1$ or below, approx.~\ref{app2}-\ref{app3} reduce to $\omega^{3/2}>>||\dot{\Gamma}(t)||,||\Gamma(t)A(t)||$.
\end{rem}

This allows to obtain the following (approximate) inequality
\begin{equation}\label{eq:bound5}
\begin{aligned}
    x(t) &\le  x_0 + \int_{t_0}^t 
            \left[
                A - k \alpha B B^T \Gamma 
            \right]x d\tau
    \\
    & + \frac{2 (\alpha + k)}{\sqrt{\omega}} |B|
    + \frac{k (\alpha+k)}{\omega}|B|^2 
    \max_t \left\{
        \frac{\bar{\sigma}(\Gamma(t))}{\sqrt{\underline{\sigma}(\Gamma(t))}} 
    \right\}\,,
\end{aligned}
\end{equation}
where $\bar{\sigma}(\Gamma(t))$ is the maximum eigenvalue of $\Gamma(t)$, $\underline{\sigma}(\Gamma(t))$ is its minimum eigenvalue and we used the fact (given without demonstration) that
\[
    x^T \Gamma x < 1 \implies |\Gamma x| \le \frac{\bar{\sigma}(\Gamma(t))}{\sqrt{\underline{\sigma}(\Gamma(t))}}\,.
\]

Let us define $\kappa = \max_t \left\{
        \frac{\bar{\sigma}(\Gamma(t))}{\sqrt{\underline{\sigma}(\Gamma(t))}} \right\}$ for brevity (note that $\kappa$ is a measure of the hyper-ellipsoid elongation).
% 
% Thus
% \begin{equation}\label{eq:bound6}
% \begin{aligned}
%     |x(t) - \bar{x}| &\le  
%     \int_{t_0}^t 
%             \left[
%                 A - k \alpha B B^T \Gamma 
%             \right] 
%             \left|
%                 x(t)-\bar{x}(t)
%             \right|d\tau
%     \\ 
%     & + \frac{2 (\alpha + k)}{\sqrt{\omega}} |B|
%     % 
%     + \frac{k (\alpha+k)}{\omega}|B|^2 
%     \kappa\,.
% \end{aligned}
% \end{equation}
% 
% If we apply the Gronwall-Bellman lemma to this approximate dynamics, we get
%
By applying the Gronwall-Bellman lemma, we obtain
\begin{equation}\label{boundapprox}
    |x(t) - \bar{x}| \le \left\{
     \frac{2 (\alpha + k)}{\sqrt{\omega}} |B|
    + \frac{k (\alpha+k)}{\omega}|B|^2 
    \kappa
    \right\} \eta \,,
\end{equation}
% Again, define 
where we defined
$\eta = \max_t \left| \left|
    e^{\int_{t_0}^t 
    \left[
        A(\tau) - k \alpha B B^T \Gamma(\tau) 
    \right] d\tau}
    \right| \right| $.

Then, from the desired condition $|x(t) - \bar{x}| \le \Delta$ we get the following inequality in terms of $1/\sqrt{\omega}$
\begin{equation}\label{eq:wstar}
    % \left\{
    % 
     \frac{2}{\sqrt{\omega}}
    + \frac{k \kappa |B|}{\omega}
    % \right\} \eta 
    % 
    \le \frac{\Delta}{(\alpha + k) \eta |B|}\,.
\end{equation}

Solving~\eqref{eq:wstar} for $\omega$ provides an indication of the minimum dithering frequency needed by the algorithm.

\begin{rem} In approximation~\ref{app2} we neglected the terms proportional to $\frac{1}{\omega\sqrt{\omega}}$. If $\omega$ is large enough so that also the terms $\propto \frac{1}{\omega}$ are negligible in~\eqref{eq:bound4}, expression~\eqref{eq:wstar} admits a neat interpretation. Indeed, it can be rewritten as
\begin{equation}\label{eq:1storder}
     \omega \ge 
     \left[
        \frac{2(\alpha + k) |B| \eta}{\Delta}
    \right]^2\,,
\end{equation}

i.e., the square-root of the minimum dithering frequency is inversely proportional to the required maximum distance, and is directly proportional to the terms that influence the amplitude of the perturbation injected into the system ($\alpha, k, |B|$). Moreover, a larger $\omega$ is needed if the system exhibits growing modes which tends to amplify an initial perturbation, whose behaviour is concisely captured by $\eta$. 
\end{rem}

%%%%%%%%%%%%%%%%%%%%%%%%%%%%%%%%%%%%%%%%%%%%%%%%%%%%%%%%%%%%%%
\section{Examples}\label{sec:examples}

In this section we consider two numerical examples to show the effectiveness of the proposed approach for finite-time stabilization via extremum seeking.

\begin{example}\label{ex:Ex1}
Let us consider the following second order LTI system
\begin{equation*}\label{eq:sysex1}
\begin{aligned}
    \dot{x}(t) &= Ax+Bu \\
            &=
        \begin{bmatrix}
         0   & 0.01 \\
        -0.1 & 0.15 
        \end{bmatrix} x(t) + 
        \begin{bmatrix}
         0  \\
         1  
        \end{bmatrix} u(t)
\end{aligned}
\end{equation*}
Where the input $u(t)$ is chosen as in~\eqref{eq:ESinput}. We search for the values of the control parameters $k, \alpha, \omega$ which make~\eqref{eq:sysex1} FTS with respect to
\begin{itemize}
    \item $R = \mathbf{I_2} / 0.4$
    \item $\Gamma(t) = \Gamma = \mathbf{I_2} / 0.5$
    \item $t_0 = 0$
    \item $T = 10$
\end{itemize}
\noindent
where $\mathbf{I_n}$ is the identity matrix of order $n$. The maximum allowed distance between $x(t)$ and $\bar{x}(t)$ has been set to~$\Delta = 0.09$ and we have chosen $\Pi = \Gamma$.

\end{example}

The associated DLMI problem~\eqref{eq:FTScond} has been discretized with a time step of $T_s = 0.1$~s, with $Q(t)$ assumed to be piecewise-linear, solved in Matlab using the YALMIP~\cite{yalmip} parser and the MOSEK~\cite{mosek} solver. 
To solve the problem, which is non-linear, a scan in the product $k\alpha$ was performed (starting at $k\alpha = 0$ with a step of $0.01$) to find the minimum value of $k\alpha$ which makes the problem feasible in $Q(t)$. The variables $k$ and $\alpha$ were assumed to be constant and equal. 

For this problem, we obtained the solution $k\alpha = 0.04$, with the resulting $\omega_{\min} \cong 750$~rad/s obtained from~\eqref{eq:wstar}. For comparison, condition~\eqref{eq:1storder} gives a very similar value of $\omega_{\min} \cong 739$~rad/s. 
The fact that the first and second order approximated conditions~\eqref{eq:1storder}-\eqref{eq:wstar} yield a very similar value for $\omega_{\min}$ suggests that the error introduced by approximations~\ref{app1}-\ref{app2} is negligible (note that, in this case, $\dot{\Gamma}(t) = 0$).

Figs.~\ref{fig:ex1-1} and~\ref{fig:ex1-4} show the obtained results for $k = \alpha, \omega = \omega_{\min}$ and five different random choices of the initial state, all such that $x_0^T R x_0 > 0.8$. 
In all the considered cases, the distance between the closed-loop and the averaged dynamics is well below the chosen threshold $\Delta$.

Finally, it is worth remarking again here that this approach still works even when the control direction is reversed, making it appealing for systems with unknown control direction.
\begin{figure}
 \centering
 \includegraphics[width=0.7\linewidth, trim=20 20 20 0]{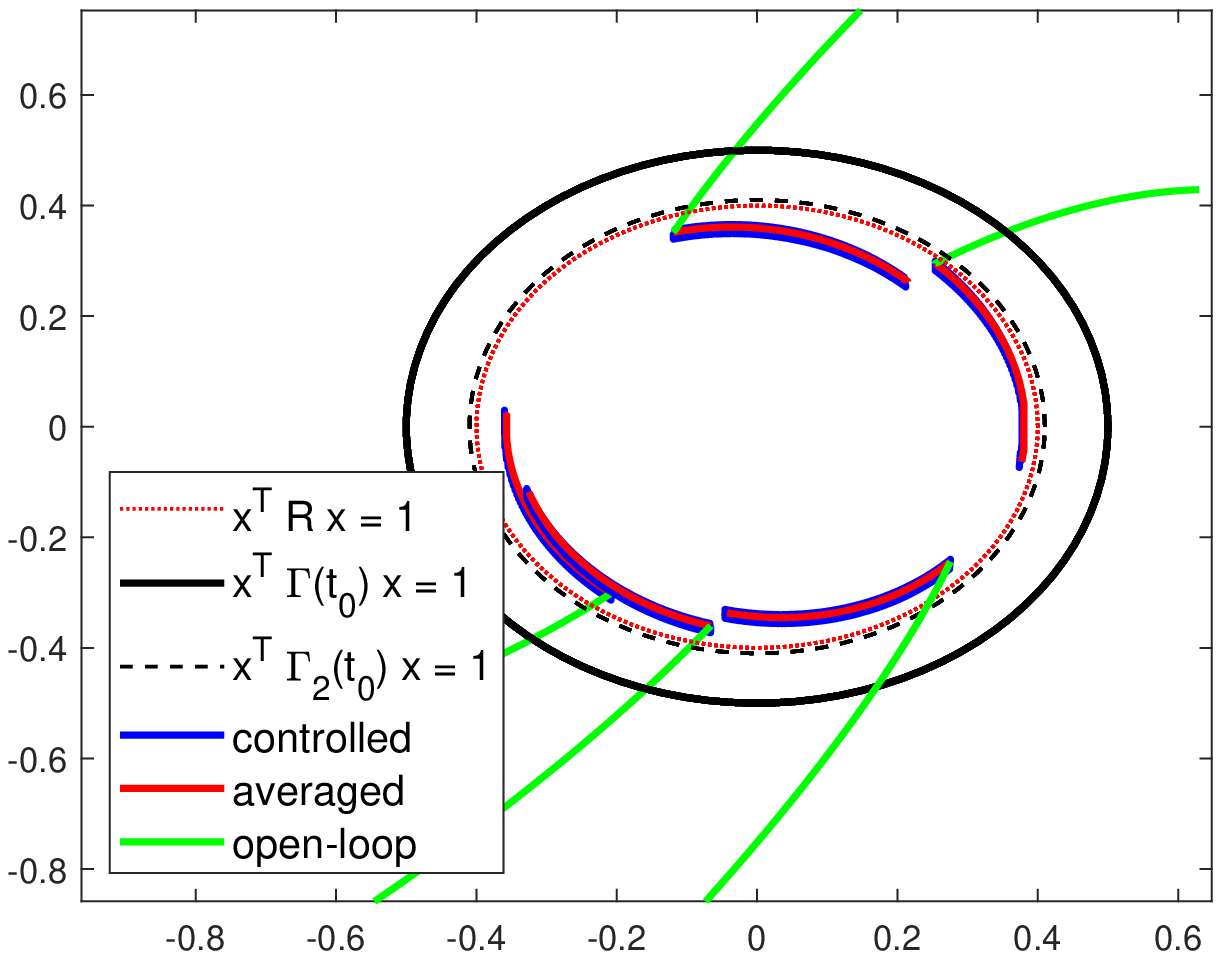}
 \caption{State trajectory of the open-loop, controlled and averaged system for $x_0 = [0.25 \quad 0.25]^T$. It can be seen that the open-loop trajectory (green line) is not FTS wrt the chosen $\Gamma\,, R\,, T\,,$ and~$t_0$. The solid black and dashed black traces represent the ellipses associated to $\Gamma$ and $\tilde{\Gamma}$, while the dotted red circle represent the ellipse defined by $R$.}
 \label{fig:ex1-1}
\end{figure}
\begin{figure}
 \centering
 \includegraphics[width=0.8\linewidth]{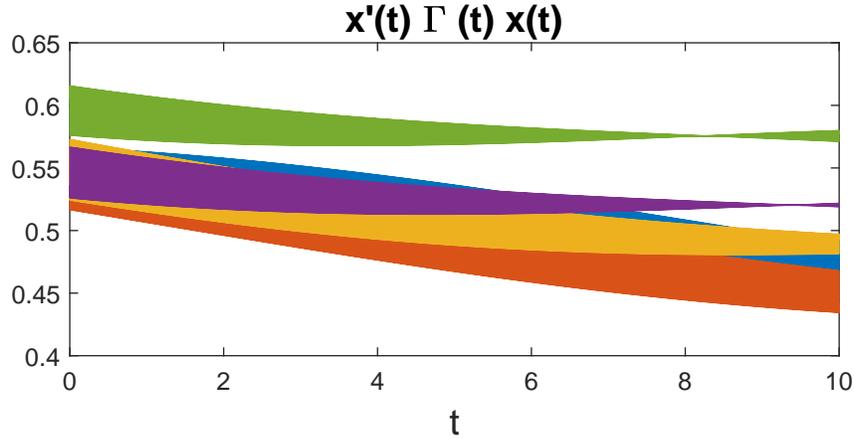}
 \caption{The product $x(t)^T,\Gamma(t)x(t)$ for different choices of the initial state.}
 \label{fig:ex1-4}
\end{figure}

\begin{example}\label{ex:Ex2} 

In this example, we consider again the FTS problem of Example~\ref{ex:Ex1}, but this time the B matrix is given by:
\[
    B(t) = 
    \begin{bmatrix}
     0  \\
     \cos(\frac{2\pi}{T}t)
    \end{bmatrix}\,.
\]
\end{example}

The B matrix is time-varying, with a loss of controllability at $t = 2.5$ and at $t = 7.5$ where $B(t) = [0 \quad 0]^T$. Problem~\eqref{eq:ES-FTS} was solved using the MOSEK~\cite{mosek} solver discretizing the DLMI condition with a sampling time $T_s = 0.01$~s. The problem admits a solution for $k\alpha = 0.11$.

% \textcolor{red}{
Although an explicit bound in the case of time-varying $B(t)$ was not derived in §\ref{sec:bounds}, if $B(t)$ varies on time-scales which are slower than those of the dithering/mixing signals, if approx.~\ref{app1} holds for $|\dot{B}(t)|$ we expect~\eqref{eq:wstar} to still provide a good approximation for $\omega_{\min}$. 
For this example, the value obtained by~\eqref{eq:wstar} is $\omega_{\min} \cong 1931$ rad/s, and again~\eqref{eq:1storder} provides a very close value of about $1902$ rad/s.
% }

Figs.~\ref{fig:ex2-1} and~\ref{fig:ex2-4} show the obtained results for $k = \alpha, \omega = \omega_{\min}$ for one random choice of the initial state.
\begin{figure}
 \centering
 \includegraphics[width=0.65\linewidth]{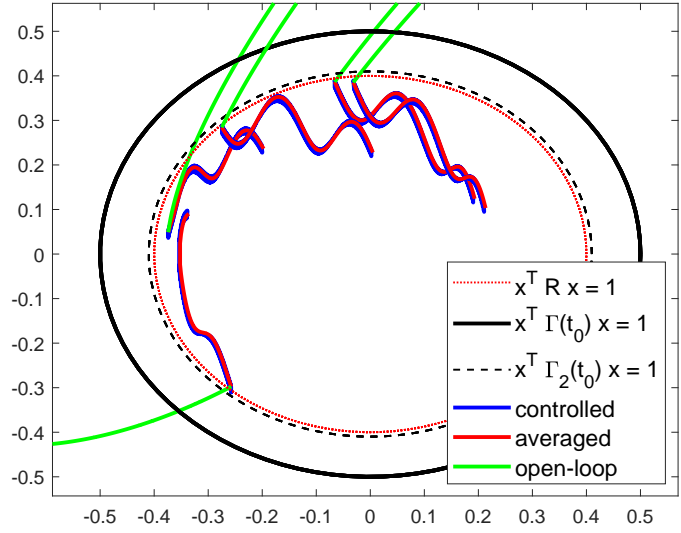}
 \caption{State trajectory of the open-loop, controlled and averaged system. It can be seen that the open-loop trajectory (green line) is not FTS wrt the chosen $\Gamma\,, R\,, T\,,$ and~$t_0$. The solid black and dashed black traces represent the ellipses associated to $\Gamma$ and $\tilde{\Gamma}$, while the dotted red circle represent the ellipse defined by $R$.}
 \label{fig:ex2-1}
\end{figure}
\begin{figure}
 \centering
 \includegraphics[width=0.8\linewidth]{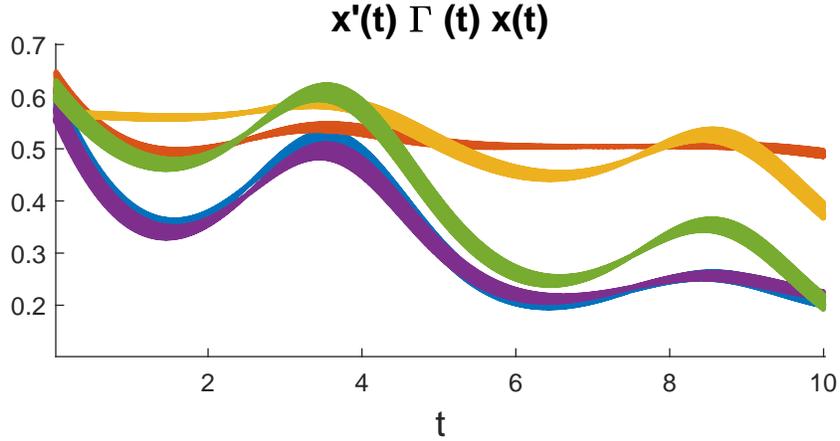}
 \caption{The product $x(t)^T,\Gamma(t)x(t)$ for five random choice of the initial state.}
 \label{fig:ex2-4}
\end{figure}

\begin{example}\label{ex:Ex3}
Consider the LTV system
\begin{equation*}
\begin{aligned}
    \dot{x}(t) &= A(t)x+Bu \\
            &= \left(1+t/10\right)
        \begin{bmatrix}
         0.5 & -0.1 \\
         0   & -0.15 
        \end{bmatrix}  x(t) + 
        \begin{bmatrix}
         1  \\
         0  
        \end{bmatrix} u(t)
\end{aligned}
\end{equation*}
where the input $u(t)$ is chosen as in~\eqref{eq:ESinput}. We search for the values of the control parameters $k, \alpha, \omega$ which finite-time stabilize~\eqref{eq:sysex1} with respect to
\begin{itemize}
    \item $R = \begin{bmatrix}
                    6.25 & 0 \\
                    0    & 9.375 
                \end{bmatrix}$
    \item $\Gamma(t) = \Gamma_0 \left( e^{t/10}\right)$ 
            with $\Gamma_0 =\begin{bmatrix}
                    4 & 0 \\
                    0 & 6 
                \end{bmatrix}$
    \item $t_0 = 0$
    \item $T = 5$
\end{itemize}
\noindent
The maximum allowed distance between $x(t)$ and $\bar{x}(t)$ has been set to~$\Delta = 0.0735$ and we have chosen $\Pi(t) = \Gamma(t)$.

\end{example}

To solve the resulting DLMI, the time interval $[t_0, t_0 + T]$ was discretized in 300 sub-intervals; $Q(t)$ was again assumed to be piecewise-linear.

For this problem, we obtained the solution $k\alpha = 0.14$, $\omega_{\min} \cong 1714$~rad/s obtained from~\eqref{eq:wstar}. Condition~\eqref{eq:1storder} yields a very similar value of $\omega_{\min} \cong 1656$~rad/s. It can be verified that approx.~\ref{app1}-\ref{app2} are well satisfied for these values of~$\omega$.
\begin{figure}
 \centering
 \includegraphics[width=0.8\linewidth]{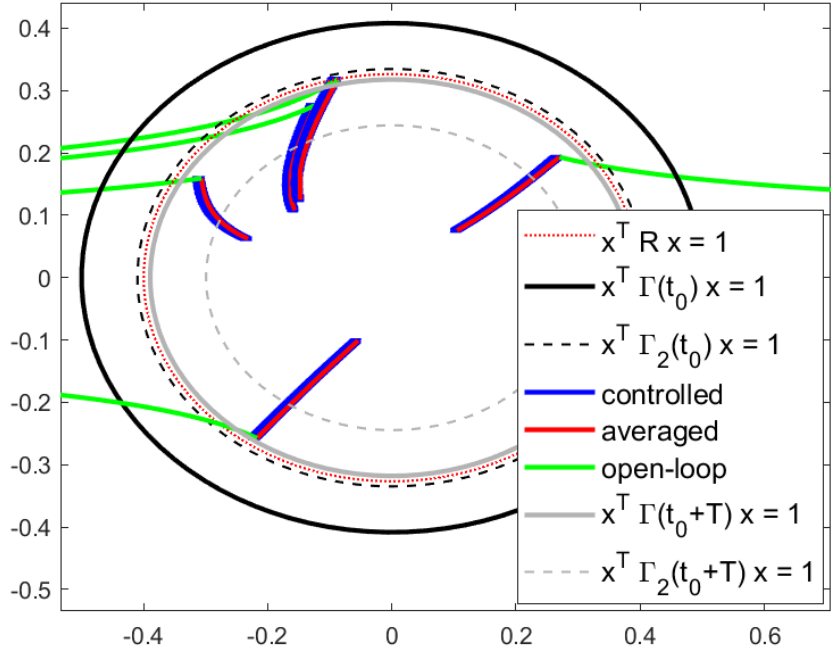}
 \caption{State trajectory of the open-loop, controlled and averaged system. It can be seen that the open-loop trajectory (green line) is not FTS wrt the chosen $\Gamma\,, R\,, T\,,$ and~$t_0$. The solid black and dashed black traces represent the ellipses associated to $\Gamma(t_0)$ and $\tilde{\Gamma}(t_0)$, the grey ones the ellipses defined by $\Gamma(t_0+T)$ and $\tilde{\Gamma)}(t_0+T$, while the dotted red circle represent the ellipse defined by $R$.}
 \label{fig:ex4-1}
\end{figure}
\begin{figure}
 \centering
 \includegraphics[width=0.7\linewidth]{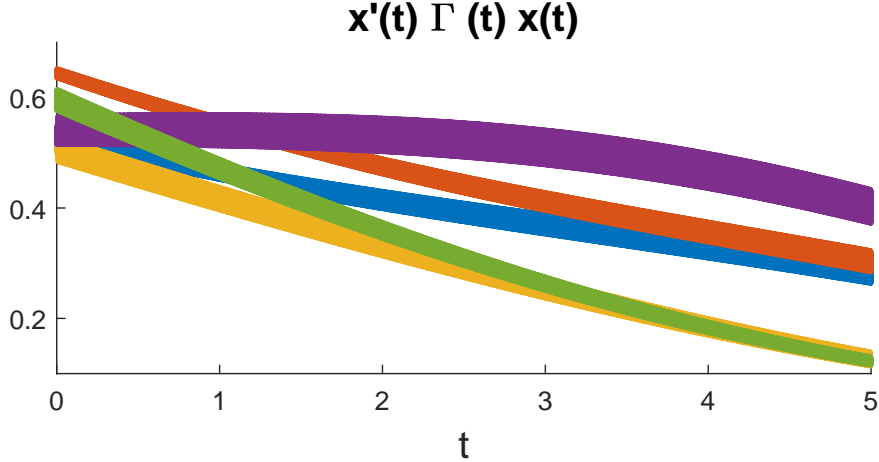}
 \caption{The product $x(t)^T,\Gamma(t)x(t)$ for five random choices of the initial state.}
 \label{fig:ex4-4}
\end{figure}

%%%%%%%%%%%%%%%%%%%%%%%%%%%%%%%%%%%%%%%%%%%%%%%%%%%%%%%%%%%%%%
\section{Conclusions}\label{sec:conclusion}

In this work, an approach for the finite-time stabilization of LTV systems with unknown control direction based on a modified version of the standard Extremum-Seeking algorithm has been presented. The proposed methodology allows to design a static state-feedback law that finite-time stabilizes the system in an average sense. This, in turn, implies the finite-time stability of the system's state trajectories under the assumption that the dithering/mixing frequency $\omega$ is chosen high enough and that the $\Gamma(t)$ matrix in the FTS definition is modified opportunely ($\tilde{\Gamma}(t)$). 
Approximate indications on the choice of a minimum dithering/mixing frequency are also given, taking advantage of the \emph{time-scale separation} property on which the ES algorithm is based to derive a lower bound on $\sqrt{\omega}$ in the form of simple first or second order inequalities. Albeit approximate, the proposed numerical examples show that this bound is indeed capable of providing a satisfactory, and sometimes even quite conservative estimate of the minimum frequency needed, which still holds when the B matrix is slowly varying over time.

\bibliographystyle{./bibliography/IEEEtran}
\bibliography{./bibliography/IEEEabrv,./bibliography/references}

\end{document}